\documentclass[a4paper]{article}

\usepackage{amsmath}
\usepackage{amssymb}
\usepackage{amsthm}
\usepackage[UKenglish]{babel}
\usepackage{Baskervaldx}
\usepackage{booktabs}
\usepackage{breakurl}
\usepackage{cite}
\usepackage{multirow}
\usepackage{tabularx}
	\newcolumntype{Y}{>{\centering\arraybackslash}X}
\usepackage[normalem]{ulem}
\usepackage{url}
\usepackage[dvipsnames]{xcolor}
	\definecolor{linkblue}{HTML}{3d25b9}

\usepackage{tikz}
	\usetikzlibrary{arrows.meta}

\usepackage{enumitem}
	\setlist{topsep=0pt,itemsep=0pt}
\usepackage[hang,flushmargin]{footmisc}
\usepackage[margin=25mm]{geometry}
\usepackage{microtype}
\usepackage{sectsty}
	\sectionfont{\large}
	\subsectionfont{\normalsize}
\usepackage{titlesec}
	\titlespacing{\section}{0pt}{12pt}{0pt}
	\titlespacing{\subsection}{0pt}{6pt}{0pt}
\usepackage{titling}
\usepackage[nottoc]{tocbibind} 
\usepackage{tocloft}

\usepackage[colorlinks=true,linkcolor=linkblue,citecolor=linkblue,urlcolor=magenta]{hyperref}
	\hypersetup{breaklinks=true}
\usepackage[capitalise,noabbrev]{cleveref}
	\crefname{equation}{equation}{equations}

\theoremstyle{plain}
	\newtheorem{theorem}{Theorem}
	\newtheorem{proposition}[theorem]{Proposition}
	
	\newtheorem{lemma}[theorem]{Lemma}
	\newtheorem{conjecture}[theorem]{Conjecture}

	\newtheorem{algorithm}[theorem]{Algorithm}

	\numberwithin{theorem}{section}
\theoremstyle{definition}
	\newtheorem{definition}[theorem]{Definition}
	\newtheorem{example}[theorem]{Example}

\newcommand{\E}{\mathcal{E}}
\renewcommand{\geq}{\geqslant}
\newcommand{\h}{\hbar}
\newcommand{\hcancel}[1]{\hbox{\sout{$#1$}}}
\newcommand{\HH}{\mathrm{Hur}}
\renewcommand{\leq}{\leqslant}

\title{The structure of Hurwitz numbers with fixed ramification profile and varying genus}
\author{Norman Do \and Jian He \and Heath Robertson}

\begin{document}

\makeatletter
\textbf{\large \thetitle}

\textbf{\theauthor}
\makeatother

School of Mathematics, Monash University, VIC 3800 Australia \\
Email: \href{mailto:norm.do@monash.edu}{norm.do@monash.edu}, \href{mailto:jian.he@monash.edu}{jian.he@monash.edu}, \href{mailto:heath.robertson@monash.edu}{heath.robertson@monash.edu}

{\em Abstract.} In 1891, Hurwitz introduced the enumeration of genus $g$, degree $d$, branched covers of the Riemann sphere with simple ramification over prescribed points and no branching elsewhere. He showed that for fixed degree $d$, the enumeration possesses a remarkable structure. More precisely, it can be expressed as a linear combination of exponentials $m^{2g-2+2d}$, where $m$ ranges over the integers from $1$ to $\binom{d}{2}$.

In this paper, we generalise this structural result to Hurwitz numbers that enumerate branched covers which also have a prescribed ramification profile over one point. Our proof fundamentally uses the infinite wedge space, in particular the connected correlators of products of $\E$-operators. The recent study of Hurwitz numbers has often focussed on their structure with fixed genus and varying ramification profile. Our main result is orthogonal to this, allowing for the explicit calculation and the asymptotic analysis of Hurwitz numbers in large genus.

We pose the broad question of which other enumerative problems exhibit analogous structure. We prove that orbifold Hurwitz numbers can also be expressed as a linear combination of exponentials and conjecture that monotone Hurwitz numbers share a similar structure, but with the inclusion of an additional linear term.

\emph{Acknowledgements.} The first author would like to thank Scott Mullane for the initial question that motivated this work. The third author undertook this research as part of an undergraduate research program in the School of Mathematics at Monash University.


\vspace{6mm} \hrule \vspace{4mm}

\tableofcontents

\vspace{6mm} \hrule

\section{Introduction}

In 1891, Hurwitz introduced the weighted enumeration $\HH_{g,d}$ of genus $g$, degree $d$, branched covers of the Riemann sphere with simple ramification over $2d+2g-2$ prescribed points and no branching elsewhere~\cite{hur1891}. The weight of a branched cover is the reciprocal of its number of automorphisms --- see \cref{def:hurwitz} below for further details. Considering a branched cover via its monodromy representation allows one to translate the problem to one of enumerating permutations. In particular, $\HH_{g,d}$ is $\frac{1}{d!}$ times the number of tuples $(\tau_1, \tau_2, \ldots, \tau_{2d+2g-2})$ of transpositions in the symmetric group $S_d$ that generate a transitive subgroup and whose product is the identity. Hurwitz discovered the following remarkable structure to this enumeration for fixed degree and varying genus.

\begin{theorem}[Hurwitz~\cite{hur1891,hur1901}] \label{thm:hurwitz}
For fixed $d \geq 2$, the function $\HH_{g,d}$ can be expressed as follows, where $B(d,m) \in \mathbb{Z}$.
\[
\HH_{g,d} = \frac{2}{d!^2} \sum_{1 \leq m \leq \binom{d}{2}} B(d, m) \cdot m^{2d+2g-2}
\]
\end{theorem}

For example, the enumeration $\HH_{g,5}$ is given by the formula
\[
\frac{2}{5!^2} \left( 10^{2g+8} - 25 \cdot 6^{2g+8} + 16 \cdot 5^{2g+8} - 100 \cdot 4^{2g+8} + 400 \cdot 3^{2g+8} + 600 \cdot 2^{2g+8} - 4000 \cdot 1^{2g+8} \right).
\]

The enumeration of branched covers or factorisations in symmetric groups more generally comes under the umbrella of Hurwitz theory, which has experienced a renaissance in recent decades. Deep relations have been discovered between these so-called Hurwitz numbers and integrability~\cite{oko00}, moduli spaces of curves~\cite{ELSV01}, Gromov--Witten theory of curves~\cite{oko-pan06}, and topological recursion~\cite{bou-mar08, eyn-mul-saf11}.

In the present work, we generalise \cref{thm:hurwitz} to the family of Hurwitz numbers defined below, which enumerate branched covers that have a prescribed ramification profile over one point. Up to possible normalisation factors, these are precisely the Hurwitz numbers that feature in the celebrated ELSV formula~\cite{ELSV01}.\footnote{Elsewhere in the literature, these are sometimes called {\em simple Hurwitz numbers}, where the word ``simple'' refers to simple ramification, and sometimes called {\em single Hurwitz numbers}, where the word ``single'' refers to the ramification profile being prescribed over one point. We will omit these adjectives and simply use the term {\em Hurwitz numbers}.}

\begin{definition} \label{def:hurwitz}
Let $g$ be an integer and $\mu = (\mu_1, \ldots, \mu_n)$ be a tuple of positive integers. The {\em Hurwitz number} $H_{g;\mu}$ is the weighted enumeration of genus $g$ connected branched covers of $\mathbb{CP}^1$ with ramification profile $\mu$ over $\infty \in \mathbb{CP}^1$, simple ramification over $|\mu| + 2g - 2 + n$ prescribed points of $\mathbb{CP}^1$, and no branching elsewhere. The preimages of $\infty \in \mathbb{CP}^1$ are labelled from $1$ to $n$ such that the point labelled $i$ has ramification index~$\mu_i$.

An {\em isomorphism} between branched covers $f: \Sigma \to \mathbb{CP}^1$ and $\tilde{f}: \tilde{\Sigma} \to \mathbb{CP}^1$ is a Riemann surface isomorphism $\phi: \Sigma \to \tilde{\Sigma}$ that satisfies $f = \tilde{f} \circ \phi$ and preserves the labels of the preimages of $\infty \in \mathbb{CP}^1$. Each branched cover is counted with weight equal to the reciprocal of its number of automorphisms.

The {\em disconnected Hurwitz number} $H_{g;\mu}^\bullet$ is obtained by allowing the branched cover to possibly be disconnected.
\end{definition}

Observe that the number of simple ramification points appearing in \cref{def:hurwitz} is determined by the Riemann--Hurwitz formula. The Hurwitz number $H_{g;\mu}$ is zero if $g$ is negative. On the other hand, the disconnected Hurwitz number $H_{g;\mu}^\bullet$ may be non-zero if $g$ is negative, as it is the Euler characteristic, rather than the genus, that is additive over disjoint union.

By a standard monodromy argument using the Riemann existence theorem, we have the following equivalent enumeration of permutations in the symmetric group.\footnote{See the book of Cavalieri and Miles for a down-to-earth explication of this argument, as well as a gentle introduction to Hurwitz theory more generally~\cite{cav-mil16}.}

\begin{proposition} \label{prop:hurwitz}
The Hurwitz number $H_{g;\mu}$ is equal to $\frac{|\mathrm{Aut}(\mu)|}{|\mu|!}$ times the number of tuples $(\tau_1, \tau_2, \ldots, \tau_{|\mu|+2g-2+n})$ of transpositions in the symmetric group $S_{|\mu|}$ that generate a transitive subgroup and whose product has cycle type $\mu$.\footnote{Here and throughout, we write $|\mu|$ for the sum of the terms of the tuple $\mu$. The notation $\mathrm{Aut}(\mu)$ denotes the group of permutations acting on the tuple $\mu = (\mu_1, \ldots, \mu_n)$ that leave it invariant.} The disconnected Hurwitz number $H_{g;\mu}^\bullet$ is obtained by dropping the transitivity constraint.
\end{proposition}

\begin{example}
The Hurwitz number $H_{g;21}$ is equal to $\frac{1}{3!}$ times the number of tuples $(\tau_1, \tau_2, \ldots, \tau_{2g+3})$ of transpositions in $S_3$ that generate a transitive subgroup and whose product has cycle type $(2,1)$. The product of any $2g+3$ transpositions in $S_3$ is an odd permutation and hence, automatically has cycle type $(2,1)$. The subgroup generated by $2g+3$ transpositions in $S_3$ is transitive as long as the $2g+3$ transpositions are not all identical. Hence, we have $H_{g;21} = \frac{1}{6} (3^{2g+3} - 3)$ for all non-negative integers $g$.
\end{example}

The main theorem of this paper is the following generalisation of \cref{thm:hurwitz}, appearing well over a century after Hurwitz's original work, yet not previously observed in the literature to the best of our knowledge.

\begin{theorem}[Structure of Hurwitz numbers] \label{thm:main}
For fixed $\mu = (\mu_1, \ldots, \mu_n)$ with sum $d \geq 2$, the function $H_{g;\mu}$ can be expressed as follows, where $C(\mu, m) \in \mathbb{Z}$.
\[
H_{g;\mu} = \frac{2}{d! \, \mu_1 \cdots \mu_n} \sum_{1 \leq m \leq \binom{d}{2}} C(\mu, m) \cdot m^{d+2g-2+n}
\]
\end{theorem}

Observe that Hurwitz's result is recovered by setting $\mu = (1, 1, \ldots, 1)$, in which case we have the equality $\HH_{g,d} = \frac{1}{d!} H_{g;11\cdots 1}$. The factor $\frac{1}{d!}$ arises due to the labelling of the preimages of $\infty \in \mathbb{CP}^1$ in \cref{def:hurwitz}. These labels could be removed from the definition at the expense of introducing factors of $\frac{1}{|\mathrm{Aut}(\mu)|}$ in various places.

Our proof of \cref{thm:main} fundamentally uses the infinite wedge space, also known in the literature as the fermionic Fock space. It thus differs substantially from the original proofs of \cref{thm:hurwitz} due to Hurwitz, who did not have access to this mathematical technology~\cite{hur1891, hur1901}.

For fixed $\mu = (\mu_1, \ldots, \mu_n)$, we introduce the exponential generating function
\[
F_\mu(\h) = \sum_{g=0}^\infty H_{g;\mu} \, \frac{\h^{|\mu|+2g-2+n}}{(|\mu|+2g-2+n)!}.
\]
The generating function $F_\mu(\h)$ can be expressed as a connected correlator of a product of $\E$-operators on the infinite wedge space. These $\E$-operators were introduced by Okounkov and Pandharipande in the context of Gromov--Witten theory of curves~\cite{oko-pan06} and were previously used to encode Hurwitz numbers as certain inner products on the infinite wedge space~\cite{joh15}. We prove that such connected correlators can be algorithmically and explicitly evaluated to polynomials in the $q$-integers $[1], [2], [3], \ldots$, which are defined by the formula\footnote{These $q$-integers should not be confused with other distinct, yet closely related, notions of $q$-integers that appear in the literature.}
\begin{equation} \label{eq:qinteger}
[k] = q^{k/2} - q^{-k/2} = e^{k\h/2} - e^{-k\h/2}, \qquad \text{where } q = e^\h.
\end{equation}
A careful analysis of the combinatorics of connected correlators then leads to the proof of \cref{thm:main}. The notions of infinite wedge space, $\E$-operators and connected correlators will be explained in due course.

The structure of Hurwitz numbers with fixed genus and varying ramification profile is essentially captured by the celebrated ELSV formula~\cite{ELSV01}. It states that
\[
H_{g;(\mu_1, \ldots, \mu_n)} = (|\mu|+2g-2+n)! \, \prod_{i=1}^n \frac{\mu_i^{\mu_i}}{\mu_i!} \, P_{g,n}(\mu_1, \ldots, \mu_n),
\]
where $P_{g,n}$ is a symmetric polynomial of degree $3g-3+n$ whose coefficients are Hodge integrals on the moduli space of curves $\overline{\mathcal M}_{g,n}$. On the other hand, \cref{thm:main} describes the structure of Hurwitz numbers with fixed ramification profile and varying genus. As a result, it allows for the explicit calculation of Hurwitz numbers in large genus and leads to the following asymptotic result.

\begin{theorem}[Large genus asymptotics of Hurwitz numbers] \label{thm:asymptotics}
For fixed $\mu = (\mu_1, \ldots, \mu_n)$ with sum $d$, we have
\[
H_{g;\mu} \sim \frac{2}{d! \, \mu_1 \cdots \mu_n} \binom{d}{2}^{d+2g-2+n} \quad \text{as } g \to \infty.
\]
\end{theorem}

It is natural to seek other enumerative problems whose structure is similar to that of Hurwitz numbers, for fixed partition and varying genus. For example, we prove that orbifold Hurwitz numbers satisfy essentially the same structure. For a fixed positive integer $r$, the $r$-orbifold Hurwitz numbers are defined similarly to Hurwitz numbers, but with an extra point with ramification profile of the form $(r, r, \ldots, r)$. They arise in the enumerative geometry of the orbifold $\mathbb{CP}^1[r]$~\cite{BHLM14, do-lei-nor16, joh-pan-tse11}. We also consider the case of monotone Hurwitz numbers, which are defined similarly to Hurwitz numbers, but with an additional monotonicity constraint. On the basis of strong numerical evidence, we conjecture that they can be expressed as a linear combination of exponentials plus a linear term, thus bearing a strong semblance to the structure of Hurwitz numbers, but with an added twist. Monotone Hurwitz numbers were introduced by Goulden, Guay-Paquet and Novak, who realised that they arise as coefficients in the large $N$ expansion of the HCIZ matrix integral over the unitary group $U(N)$~\cite{gou-gua-nov14}.

There may be many enumerative problems that exhibit the same, or similar, structure to that shown in \cref{thm:main} for Hurwitz numbers. We propose that the topological recursion may be a source of natural examples worthy of further investigation~\cite{che-eyn06, eyn-ora07}.

The structure of the remainder of the paper is as follows.

\begin{itemize}
\item In \cref{sec:wedge}, we provide a minimal introduction to the infinite wedge space and define the $\E$-operators (\cref{def:Eoperators}). We present two key facts on which our subsequent analysis heavily rests. The first is the expression of generating functions for Huriwtz numbers as correlators of $\E$-operators (\cref{prop:hurwitzwedge}) and the second is the commutation relation for the $\E$-operators (\cref{lem:Ecommutation}).

\item In \cref{sec:proof}, we introduce the notion of connected correlators for $\E$-operators (\cref{def:concorrelator}) and describe an explicit algorithm to compute them (\cref{alg:connected}). A thorough combinatorial analysis of this algorithm leads to a proof of the main theorem on the structure of Hurwitz numbers with fixed ramification profile and varying genus (\cref{thm:main}).

\item In \cref{sec:extras}, we use our main theorem to prove a result concerning the large genus asymptotics of Hurwitz numbers (\cref{thm:asymptotics}). We then briefly introduce orbifold Hurwitz numbers and prove that they satisfy essentially the same structure as the usual Hurwitz numbers (\cref{thm:orbifold}). We ask whether there are other enumerative problems with analogous structure and propose the monotone Hurwitz numbers as a natural candidate. We present numerical evidence to support our conjecture that they can be expressed as a linear combination of exponentials plus a linear term (\cref{con:monotone}).
\end{itemize}

\section{Hurwitz numbers via the infinite wedge space} \label{sec:wedge}

In this section, we describe how generating functions for disconnected Hurwitz numbers arise as correlators of $\E$-operators on the infinite wedge space. We provide a minimal, non-standard presentation of the infinite wedge space and the $\E$-operators acting on it. This is followed by statements of known facts that will be required later in the paper, devoid of any proofs. The reader seeking a more pedagogical exposition and further details is encouraged to consult the literature, particularly the paper of Johnson~\cite{joh15}.

\subsection{The infinite wedge space}

Let $\mathcal{P}$ denote the set of integer partitions, including the empty partition $\emptyset$. We think of partitions via their Young diagrams drawn in the so-called Russian orientation, as shown in~\cref{fig:wedge}. A partition $\lambda$ can be expressed in Frobenius notation as $(\alpha_1, \alpha_2, \ldots, \alpha_d \,|\, \beta_1, \beta_2, \ldots, \beta_d)$, where $d$ denotes the number of boxes on the ``diagonal'' of the Young diagram, $\alpha_i$ denotes the number of boxes to the right of the $i$th box on the diagonal that lie in the same ``row'', and $\beta_i$ denotes the number of boxes to the left of the $i$th box on the diagonal that lie in the same ``column''.

\begin{figure}[ht!]
\centering
\begin{tikzpicture}
\def\r{0.75}
\foreach \x in {(0,0), (\r,\r), (2*\r,2*\r), (3*\r,3*\r), (-\r,\r), (0,2*\r), (\r,3*\r)}
	\draw[shift = \x] (0,0) -- (\r,\r) -- (0,2*\r) -- (-\r,\r) -- cycle;
\foreach \x in {-5,-4,...,0}
	{\draw[dashed] (\x*\r,0) -- (\x*\r,-\x*\r);
	\node[below] at (\x*\r,0) {$\x$};}
\foreach \x in {1,2,...,5}
	{\draw[dashed] (\x*\r,0) -- (\x*\r,\x*\r);
	\node[below] at (\x*\r,0) {$\x$};}
\draw[thick, arrows = -Stealth] (0,0) -- (5*\r,5*\r);
\draw[thick, arrows = -Stealth] (0,0) -- (-5*\r,5*\r);
\draw[arrows = Stealth-Stealth] (-6*\r,0) -- (6*\r,0);
\node[rotate = 45, gray, left] at (5*\r,6*\r) {``rows''};
\node[rotate = -45, gray, right] at (-5*\r,6*\r) {``columns''};
\end{tikzpicture}
\caption{Depicted here is the Young diagram of the partition $(4,3)$, which can be expressed in Frobenius notation as $(3, 1 \,|\, 1, 0)$. Note that the diagram is situated with its lowest corner at the origin and is scaled so that each box has area 2.} \label{fig:wedge}
\end{figure}
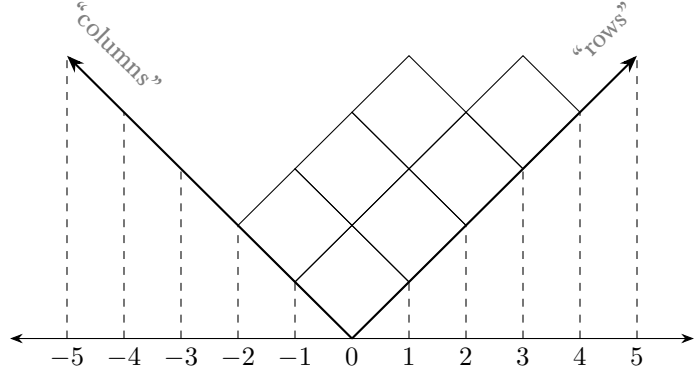

Define the {\em infinite wedge space} to be the graded vector space
\[
\Lambda = \mathrm{span}_\mathbb{C} \{ v_\lambda \,|\, \lambda \in \mathcal{P} \},
\]
where $v_\lambda$ has degree $|\lambda|$. Equip $\Lambda$ with the inner product that makes $\{ v_\lambda \,|\, \lambda \in \mathcal{P} \}$ an orthonormal basis.\footnote{To be precise, we have defined here the charge zero subspace of the infinite wedge space. It is called the infinite wedge space since the vector $v_\lambda$ can be represented as $s_{\lambda_1-\frac{1}{2}} \wedge s_{\lambda_2-\frac{3}{2}} \wedge s_{\lambda_3-\frac{5}{2}} \wedge \cdots$, an infinite wedge of vectors in $\mathrm{span}_\mathbb{C} \{ s_k \,|\, k \in \mathbb{Z} + \frac{1}{2} \}$.}

It is customary to adopt the bra-ket notation in which the basis vectors are written as $v_\lambda = | \, \lambda \, \rangle$ and their dual vectors as $\langle \, \lambda \, |$. For an operator $\mathcal{F}: \Lambda \to \Lambda$, we will write
\[
\langle \, \mathcal{F} \, \rangle = \langle \, \emptyset \,|\, \mathcal{F} \,|\, \emptyset \, \rangle.
\]
These inner products are often referred to as {\em correlators} or {\em vacuum expectation values}, terminology that derives from physics. Such expressions commonly arise in enumerative geometry~\cite{oko-pan06} and integrable systems~\cite{miw-jim-dat00}.

\subsection{The \texorpdfstring{$\E$}{E}-operators}

Disconnected Hurwitz numbers arise as certain correlators on the infinite wedge space. The crucial ingredients we require are the $\E$-operators, which were introduced by Okounkov and Pandharipande in the context of Gromov--Witten theory of curves~\cite{oko-pan06}. We provide here a combinatorial definition for the $\E$-operators, which requires the following terminology concerning Young diagrams.
\begin{itemize}
\item For $m$ a positive integer, an {\em $m$-ribbon} of a Young diagram is a connected set of $m$ squares that contains no $2 \times 2$ subsquare and whose removal leaves a Young diagram. 
\item The {\em height} of a ribbon is the number of rows that it occupies minus 1.
\item The {\em sign} of a ribbon is $-1$ to the power of its height.
\item The {\em centre} of a ribbon is the $x$-coordinate of its centre of mass, which is an element of $\frac{1}{2} \mathbb{Z}$.
\end{itemize}

\begin{definition} \label{def:Eoperators}
For $m \in \mathbb{Z}$ and $z$ a formal variable, define the linear operator $\E_m(z): \Lambda \to \Lambda$ on basis vectors as follows.
\begin{itemize}
\item For $m < 0$, $\E_m(z) \, v_\lambda$ is a linear combination of the vectors $v_\mu$, where $\mu$ is obtained by adding an $|m|$-ribbon to $\lambda$. The coefficient of $v_\mu$ is equal to the sign of the ribbon multiplied by $e^{cz}$, where $c$ is the centre of the ribbon.
\item For $m > 0$, $\E_m(z) \, v_\lambda$ is a linear combination of the vectors $v_\mu$, where $\mu$ is obtained by removing an $m$-ribbon from $\lambda$. The coefficient of $v_\mu$ is equal to the sign of the ribbon multiplied by $e^{cz}$, where $c$ is the centre of the ribbon.
\item For $m = 0$, $\E_0(z)$ is diagonal with respect to the basis $\{ v_\lambda \,|\,\lambda \in \mathcal{P} \}$. The eigenvalue of $\E_0(z)$ corresponding to $v_\lambda$ is 
\[
\sum_{i=1}^d \Big( e^{(\alpha_i+1/2)z} - e^{-(\beta_i+1/2)z} \Big) + \frac{1}{e^{z/2} - e^{-z/2}},
\]
where $\lambda$ is expressed in Frobenius notation as $(\alpha_1, \alpha_2, \ldots, \alpha_d \,|\, \beta_1, \beta_2, \ldots, \beta_d)$.
\end{itemize}
\end{definition}

\begin{figure}[ht!]
\centering
\begin{tikzpicture}
\def\r{0.2}
\def\d{0.4}
\def\x{84}
\def\y{80}
% YOUNG DIAGRAM A1
\begin{scope}[xshift = 0, yshift = \y]
\foreach \x in {(0,0), (\r,\r), (2*\r,2*\r), (3*\r,3*\r), (-\r,\r), (0,2*\r), (\r,3*\r)}
	\draw[shift = \x] (0,0) -- (\r,\r) -- (0,2*\r) -- (-\r,\r) -- cycle;
\foreach \x in {(-4*\r,4*\r), (-3*\r,3*\r), (-2*\r,2*\r)}
	\filldraw[fill = green!20, shift = \x] (0,0) -- (\r,\r) -- (0,2*\r) -- (-\r,\r) -- cycle;
\draw[thick, arrows = -Stealth] (0,0) -- (6*\r,6*\r);
\draw[thick, arrows = -Stealth] (0,0) -- (-6*\r,6*\r);
\node at (0,-\d) {$\text{sign} = 1$};
\node at (0,-2*\d) {$\text{centre} = -3$};
\end{scope}
% YOUNG DIAGRAM A2
\begin{scope}[xshift = \x, yshift = \y]
\foreach \x in {(0,0), (\r,\r), (2*\r,2*\r), (3*\r,3*\r), (-\r,\r), (0,2*\r), (\r,3*\r)}
	\draw[shift = \x] (0,0) -- (\r,\r) -- (0,2*\r) -- (-\r,\r) -- cycle;
\foreach \x in {(-3*\r,3*\r), (-2*\r,2*\r), (-\r,3*\r)}
	\filldraw[fill = green!20, shift = \x] (0,0) -- (\r,\r) -- (0,2*\r) -- (-\r,\r) -- cycle;
\draw[thick, arrows = -Stealth] (0,0) -- (6*\r,6*\r);
\draw[thick, arrows = -Stealth] (0,0) -- (-6*\r,6*\r);
\node at (0,-\d) {$\text{sign} = -1$};
\node at (0,-2*\d) {$\text{centre} = -2$};
\end{scope}
% YOUNG DIAGRAM A3
\begin{scope}[xshift = 2*\x, yshift = \y]
\foreach \x in {(0,0), (\r,\r), (2*\r,2*\r), (3*\r,3*\r), (-\r,\r), (0,2*\r), (\r,3*\r)}
	\draw[shift = \x] (0,0) -- (\r,\r) -- (0,2*\r) -- (-\r,\r) -- cycle;
\foreach \x in {(-2*\r,2*\r), (-\r,3*\r), (0,4*\r)}
	\filldraw[fill = green!20, shift = \x] (0,0) -- (\r,\r) -- (0,2*\r) -- (-\r,\r) -- cycle;
\draw[thick, arrows = -Stealth] (0,0) -- (6*\r,6*\r);
\draw[thick, arrows = -Stealth] (0,0) -- (-6*\r,6*\r);
\node at (0,-\d) {$\text{sign} = 1$};
\node at (0,-2*\d) {$\text{centre} = -1$};
\end{scope}
% YOUNG DIAGRAM A4
\begin{scope}[xshift = 3*\x, yshift = \y]
\foreach \x in {(0,0), (\r,\r), (2*\r,2*\r), (3*\r,3*\r), (-\r,\r), (0,2*\r), (\r,3*\r)}
	\draw[shift = \x] (0,0) -- (\r,\r) -- (0,2*\r) -- (-\r,\r) -- cycle;
\foreach \x in {(2*\r,4*\r), (3*\r,5*\r), (4*\r,4*\r)}
	\filldraw[fill = green!20, shift = \x] (0,0) -- (\r,\r) -- (0,2*\r) -- (-\r,\r) -- cycle;
\draw[thick, arrows = -Stealth] (0,0) -- (6*\r,6*\r);
\draw[thick, arrows = -Stealth] (0,0) -- (-6*\r,6*\r);
\node at (0,-\d) {$\text{sign} = -1$};
\node at (0,-2*\d) {$\text{centre} = 3$};
\end{scope}
% YOUNG DIAGRAM A5
\begin{scope}[xshift = 4*\x, yshift = \y]
\foreach \x in {(0,0), (\r,\r), (2*\r,2*\r), (3*\r,3*\r), (-\r,\r), (0,2*\r), (\r,3*\r)}
	\draw[shift = \x] (0,0) -- (\r,\r) -- (0,2*\r) -- (-\r,\r) -- cycle;
\foreach \x in {(4*\r,4*\r), (5*\r,5*\r), (6*\r,6*\r)}
	\filldraw[fill = green!20, shift = \x] (0,0) -- (\r,\r) -- (0,2*\r) -- (-\r,\r) -- cycle;
\draw[thick, arrows = -Stealth] (0,0) -- (6*\r,6*\r);
\draw[thick, arrows = -Stealth] (0,0) -- (-6*\r,6*\r);
\node at (0,-\d) {$\text{sign} = 1$};
\node at (0,-2*\d) {$\text{centre} = 5$};
\end{scope}
% YOUNG DIAGRAM B1
\begin{scope}[xshift = \x]
\foreach \x in {(0,0), (\r,\r), (2*\r,2*\r), (3*\r,3*\r), (-\r,\r), (0,2*\r), (\r,3*\r)}
	\draw[shift = \x] (0,0) -- (\r,\r) -- (0,2*\r) -- (-\r,\r) -- cycle;
\foreach \x in {(-\r,\r), (0,2*\r), (\r,3*\r)}
	\filldraw[fill = red!20, shift = \x] (0,0) -- (\r,\r) -- (0,2*\r) -- (-\r,\r) -- cycle;
\draw[thick, arrows = -Stealth] (0,0) -- (6*\r,6*\r);
\draw[thick, arrows = -Stealth] (0,0) -- (-6*\r,6*\r);
\node at (0,-\d) {$\text{sign} = 1$};
\node at (0,-2*\d) {$\text{centre} = 0$};
\end{scope}
% YOUNG DIAGRAM B2
\begin{scope}[xshift = 3*\x]
\foreach \x in {(0,0), (\r,\r), (2*\r,2*\r), (3*\r,3*\r), (-\r,\r), (0,2*\r), (\r,3*\r)}
	\draw[shift = \x] (0,0) -- (\r,\r) -- (0,2*\r) -- (-\r,\r) -- cycle;
\foreach \x in {(\r,3*\r), (2*\r,2*\r), (3*\r,3*\r)}
	\filldraw[fill = red!20, shift = \x] (0,0) -- (\r,\r) -- (0,2*\r) -- (-\r,\r) -- cycle;
\draw[thick, arrows = -Stealth] (0,0) -- (6*\r,6*\r);
\draw[thick, arrows = -Stealth] (0,0) -- (-6*\r,6*\r);
\node at (0,-\d) {$\text{sign} = -1$};
\node at (0,-2*\d) {$\text{centre} = 2$};
\end{scope}
\end{tikzpicture}
\caption{The top row shows the five Young diagrams that can be obtained by adding a 3-ribbon to the Young diagram for the partition $(4,3)$. The bottom row shows the two Young diagrams that can be obtained by removing a 3-ribbon from the Young diagram for the partition $(4,3)$. In each case, the sign and the centre of the ribbon have been indicated.}
\label{fig:ribbons}
\end{figure}
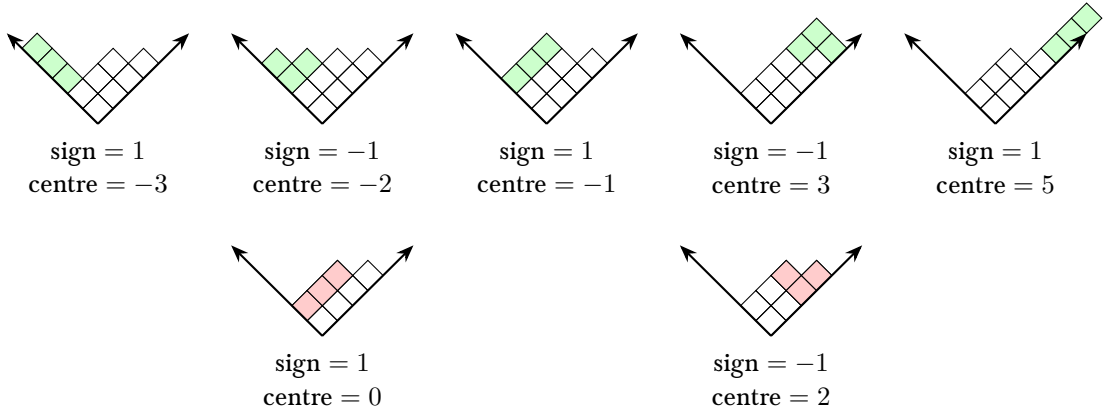

\begin{example}
One can read off $\E_{-3}(z) \, v_{43}$ and $\E_3(z) \, v_{43}$ from the information in the top and bottom rows of \cref{fig:ribbons}, respectively.
\begin{align*}
\E_{-3}(z) \, v_{43} &= e^{-3z} \, v_{43111} - e^{-2z} \, v_{4321} + e^{-z} \, v_{433} - e^{3z} \, v_{55} + e^{5z} \, v_{73} \\
\E_3(z) \, v_{43} &= v_{4} - e^{2z} \, v_{22}
\end{align*}
One can calculate $\E_0(z) \, v_{43}$ by expressing the partition $(4,3)$ in Frobenius notation as $(3, 1 \,|\, 1, 0)$.
\[
\E_0(z) \, v_{43} = \Big( e^{7z/2} + e^{3z/2} - e^{-z/2} - e^{-3z/2} + \frac{1}{e^{z/2} - e^{-z/2}} \Big) \, v_{43}
\]
\end{example}

It is immediate from \cref{def:Eoperators} that the operator $\E_m(z)$ is homogeneous of degree $-m$ with respect to the grading on $\Lambda$. As a consequence, we have the following simple evaluation for the correlator of one $\E$-operator.
\begin{equation} \label{eq:Eevaluation}
\langle \, \E_a(b\h) \, \rangle = \frac{\delta_{a,0}}{[b]}
\end{equation}

Our calculation of Hurwitz numbers via the infinite wedge space hinges on two key facts: the expression of generating functions for Hurwitz numbers as correlators of $\E$-operators and a remarkable commutation relation for the $\E$-operators. We conclude the section with statements of these results.

\begin{proposition}[Johnson~\cite{joh15}] \label{prop:hurwitzwedge}
For fixed $\mu = (\mu_1, \ldots, \mu_n)$, the exponential generating function for disconnected Hurwitz numbers\footnote{The use of $\h$ as the formal parameter here is due to the very loose analogy with Planck's constant in quantum physics and the relation $q = e^\h$ appearing in the definition of the $q$-integers of \cref{eq:qinteger} and elsewhere in quantum mathematics.}
\[
F_\mu^\bullet(\h) = \sum_{g = -\infty}^\infty H_{g;\mu}^\bullet \, \frac{\h^{|\mu|+2g-2+n}}{(|\mu|+2g-2+n)!}
\]
can be expressed as the correlator
\[
F_\mu^\bullet(\h) = \frac{1}{|\mu|! \, \mu_1 \cdots \mu_n} \, \langle \, \E_1(\h)^{|\mu|} \, \E_{-\mu_1}(0) \cdots \E_{-\mu_n}(0) \, \rangle.
\]
\end{proposition}

\begin{lemma}[Okounkov and Pandharipande~\cite{oko-pan06}] \label{lem:Ecommutation}
The $\E$-operators satisfy the commutation relation
\[
[\E_{a_1}(b_1\h), \E_{a_2}(b_2\h)] = [a_1b_2-a_2b_1] \, \E_{a_1+a_2}((b_1+b_2)\h).
\]
\end{lemma}

We often use the commutation relation of \cref{lem:Ecommutation} in the form
\[
\E_{a_1}(b_1\h) \, \E_{a_2}(b_2\h) = \E_{a_2}(b_2\h) \, \E_{a_1}(b_1\h) + [a_1b_2-a_2b_1] \, \E_{a_1+a_2}((b_1+b_2)\h).
\]
In this case, we refer to the first term of the right side as the {\em passing term} and the second as the {\em commutation term}.

It would be remiss not to say a few words about the proof of \cref{prop:hurwitzwedge} at this stage. We know from \cref{prop:hurwitz} that disconnected Hurwitz numbers enumerate factorisations of permutations into transpositions. This can be expressed in terms of irreducible characters of symmetric groups via the Burnside--Frobenius formula~\cite{cav-mil16, lan-zvo04}.\footnote{Some sources in the literature attribute the result to Burnside and others to Frobenius.} Irreducible characters of symmetric groups can be naturally expressed through the calculus of the bosonic operators $\alpha_m = \E_m(0)$ on the infinite wedge space. Such operators can be interpreted as adding and removing ribbons from Young diagrams in a way that mimics the Murnaghan--Nakayama rule for symmetric group characters. The fact that we seek to express permutations as products of transpositions leads to a deformation of the bosonic operators to $\E$-operators, resulting in the statement of \cref{prop:hurwitzwedge} for the disconnected Hurwitz number generating function. The interested reader can find the details in the literature~\cite{cav-mil16, joh15}.

Thus, the path from disconnected Hurwitz numbers to the infinite wedge space is relatively straightforward and perhaps the only surprise is that \cref{prop:hurwitzwedge} is so compact and amenable to analysis via the commutation relation of \cref{lem:Ecommutation}.

\section{The structure of Hurwitz numbers} \label{sec:proof}

In this section, we prove \cref{thm:main}, which states that Hurwitz numbers with fixed ramification profile can be expressed as a linear combination of exponentials. We introduce the notion of connected correlators of $\E$-operators and show that they can be used to calculate generating functions for connected Hurwitz numbers. The proof then proceeds with a thorough analysis of an algorithm that we present to compute such correlators.

\subsection{Combinatorics of correlators} \label{subsec:correlators}

\cref{prop:hurwitzwedge} expresses Hurwitz numbers via the infinite wedge space and motivates us to calculate correlators of the form
\[
\langle \, \E_{a_1}(b_1\h) \, \E_{a_2}(b_2\h) \cdots \E_{a_m}(b_m\h) \, \rangle,
\]
for integers $a_1, a_2, \ldots, a_m, b_1, b_2, \ldots, b_m$ and a formal variable $\h$. This can be achieved by repeatedly applying the commutation relation of \cref{lem:Ecommutation} and using \cref{eq:Eevaluation} to evaluate the correlator of a single $\E$-operator. We formalise this via the following algorithm, which appeared previously in the context of double Hurwitz numbers~\cite{joh15}.

\begin{algorithm}[Disconnected Hurwitz numbers] \label{alg:disconnected}
Let $a_1, a_2, \ldots, a_m$ be integers and $b_1, b_2, \ldots, b_m$ be non-negative integers such that if $a_i \geq 0$, then $b_i > 0$.\footnote{The integers $a_1, a_2, \ldots, a_m$ and $b_1, b_2, \ldots, b_m$ are required to satisfy some conditions to avoid the appearance of the undefined operator $\E_0(0)$. These conditions also allow us to avoid the case of commuting two $\E$-operators with zero argument, in which case one would need to use the commutation relation $[\E_m(0), \E_n(0)] = \lim_{\h \to 0} [\E_m(0), \E_n(\h)] = \delta_{m+n,0} \, m$.} Correlators of the form $\langle \, \E_{a_1}(b_1\h) \, \E_{a_2}(b_2\h) \cdots \E_{a_m}(b_m\h) \, \rangle$ can be evaluated by the following algorithm.
\begin{enumerate}[label={(\arabic*)}]
\item If $a_1 + a_2 + \cdots + a_m \neq 0$, return 0. \\
(This follows from the homogeneity properties of the $\E$-operators.)
\item If $a_1 + a_2 + \cdots + a_i < 0$ for some $1 \leq i < m$, return 0. \\
(This also follows from the homogeneity properties of the $\E$-operators and the fact that the infinite wedge space contains no vectors of negative degree.)
\item If $a_1 = a_2 = \cdots = a_m = 0$, return
\[
\frac{1}{[b_1] \, [b_2] \cdots [b_n]}.
\]
(Recall from \cref{def:Eoperators} that $E_0(b\h)$ is diagonal and from \cref{eq:Eevaluation} its eigenvalue corresponding to $v_\emptyset$.)
\item Otherwise, let $\E_{a_i}(b_i\h)$ be the rightmost $\E$-operator with positive subscript and return
\begin{align*}
& \langle \, \E_{a_1}(b_1\h) \cdots \E_{a_{i+1}}(b_{i+1}\h) \, \E_{a_i}(b_i\h) \cdots \E_{a_m}(b_m\h) \, \rangle \\
& \qquad + [a_ib_{i+1}-a_{i+1}b_i] \, \langle \, \E_{a_1}(b_1\h) \cdots \E_{a_i+a_{i+1}}((b_i+b_{i+1})\h) \cdots \E_{a_m}(b_m\h) \, \rangle.
\end{align*}
(Recall the commutation relation of \cref{lem:Ecommutation}.)
\end{enumerate}
\end{algorithm}

\cref{alg:disconnected} does indeed terminate, since the first term of the equation in Step~(4) (the {\em passing term}) moves the rightmost $\E$-operator with positive subscript further to the right, while the second term (the {\em commutation term}) has one less $\E$-operator. In both cases, the complexity of the correlator is reduced.

\begin{example}
We apply \cref{alg:disconnected} to calculate $\langle \, \E_{a_1}(b_1\h) \, \E_{a_2}(b_2\h) \, \E_{-a_1-a_2}(b_3\h) \, \rangle$ as follows, where we assume that $a_1, a_2, b_1, b_2, b_3$ are positive integers. Any term that is zero due to Step~(2) in the algorithm has been crossed out.
\begin{align*}
& \langle \, \E_{a_1}(b_1\h) \, \E_{a_2}(b_2\h) \, \E_{-a_1-a_2}(b_3\h) \, \rangle \\
={}& \hcancel{\langle \, \E_{a_1}(b_1\h) \, \E_{-a_1-a_2}(b_3\h) \, \E_{a_2}(b_2\h) \, \rangle} + [a_2b_3+a_1b_2+a_2b_2] \, \langle \, \E_{a_1}(b_1\h) \, \E_{-a_1}((b_2+b_3)\h) \, \rangle \\
={}& [a_2b_3+a_1b_2+a_2b_2] \Big( \hcancel{\langle \, \E_{-a_1}((b_2+b_3)\h) \, \E_{a_1}(b_1\h) \, \rangle} + [a_1b_2 + a_1b_3 + a_1b_1] \, \langle \, \E_{0}((b_1+b_2+b_3)\h) \, \rangle \Big) \\
={}& \frac{[a_1b_2 + a_2b_2 + a_2b_3] \, [a_1b_1 + a_1b_2 + a_1b_3]}{[b_1 + b_2 + b_3]}
\end{align*}

It is worth remarking that the correlator $\langle \, \E_{a_1}(b_1\h) \, \E_{a_2}(b_2\h) \, \E_{a_3}(b_3\h) \, \rangle$ could have been computed differently, such as by applying the commutation relation to the left two $\E$-operators before proceeding with \cref{alg:disconnected}. This would produce the expression
\[
\frac{[a_1b_1 + a_1b_3 + a_2b_1] \, [a_2b_1 + a_2b_2 + a_2b_3] + [a_1b_2 - a_2b_1] \, [(a_1+a_2)(b_1+b_2+b_3)]}{[b_1 + b_2 + b_3]},
\]
which is necessarily equal to the one derived above. Indeed, the equality of these two expressions follows from the $q$-integer relation
\[
[a] \, [b-c] + [b] \, [c-a] + [c] \, [a-b] = 0,
\]
which holds for all integers $a, b, c$. More generally, expressions involving $q$-integers are not unique and are subject to the relation above.
\end{example}

\cref{prop:hurwitzwedge} involves disconnected Hurwitz numbers, from which their connected counterparts can be derived via the inclusion--exclusion formula. An arguably more elegant approach is to make use of connected correlators for $\E$-operators, which were used in the context of double Hurwitz numbers with completed cycles by Shadrin, Spitz and Zvonkine~\cite[Section~4.4]{sha-spi-zvo12}. We explain the concept here using the setting and notation of the present work.

\begin{definition} \label{def:concorrelator}
\cref{alg:disconnected} ultimately expresses each correlator $\langle \, \E_{a_1}(b_1\h) \, \E_{a_2}(b_2\h) \cdots \E_{a_m}(b_m\h) \, \rangle$ as a sum of terms of the form $[k_1] \, [k_2] \cdots [k_\ell] \, \langle \, \E_0(c_1\h) \, \E_0(c_2\h) \cdots \E_0(c_r\h) \, \rangle$. Define the {\em connected correlator} $\langle \, \E_{a_1}(b_1\h) \, \E_{a_2}(b_2\h) \cdots \E_{a_m}(b_m\h) \, \rangle^\circ$ to be the result of applying \cref{alg:disconnected} but only retaining terms of the form above with $r = 1$.
\end{definition}

The connected correlator is named so due to the following connected analogue of \cref{prop:hurwitzwedge}.

\begin{proposition} \label{prop:conhurwitzwedge}
For fixed $\mu = (\mu_1, \ldots, \mu_n)$, the exponential generating function for connected Hurwitz numbers can be expressed as the connected correlator
\[
F_\mu(\h) = \sum_{g = 0}^\infty H_{g;\mu} \, \frac{\h^{|\mu|+2g-2+n}}{(|\mu|+2g-2+n)!} = \frac{1}{|\mu|! \, \mu_1 \cdots \mu_n} \, \langle \, \E_1(\h)^{|\mu|} \, \E_{-\mu_1}(0) \cdots \E_{-\mu_n}(0) \, \rangle^\circ.
\]
\end{proposition}

\begin{proof}
The following relation expresses the fact that a possibly disconnected branched cover is simply a disjoint union of connected branched covers.
\begin{equation} \label{eq:Finclusionexclusion}
F_\mu^\bullet(\h) = \sum_{S_1 \sqcup \cdots \sqcup S_m = \{1, 2, \ldots, n\}} \prod_{i=1}^m F_{\mu_{S_i}}(\h)
\end{equation}
The summation is over all the ways to partition the indices $\{1, 2, \ldots, n\}$ into disjoint subsets and for a subset $S = \{s_1, s_2, \ldots, s_k\}$, we define $\mu_S = (\mu_{s_1}, \mu_{s_2}, \ldots, \mu_{s_k})$.

For $n = 1$ in which case $\mu = (\mu_1)$, we have
\[
F_{\mu_1}(\h) = F_{\mu_1}^\bullet(\h) = \frac{1}{\mu_1! \, \mu_1} \, \langle \, \E_1(\h)^{\mu_1} \, \E_{-\mu_1}(0) \, \rangle = \frac{1}{\mu_1! \, \mu_1} \, \langle \, \E_1(\h)^{\mu_1} \, \E_{-\mu_1}(0) \, \rangle^\circ.
\]
The first equality holds since branched covers with a full cyclic ramification profile are necessarily connected; the second equality follows from \cref{prop:hurwitzwedge}; and the third equality follows from degree considerations in the process of applying \cref{alg:disconnected}.

By induction on $n$ and invoking the inclusion--exclusion principle, we obtain the desired result from the following relation between disconnected and connected correlators, which mimics \cref{eq:Finclusionexclusion}.
\begin{equation} \label{eq:Einclusionexclusion}
\langle \, \E_1(\h)^{|\mu|} \, \E_{-\mu_1}(0) \cdots \E_{-\mu_n}(0) \, \rangle = \sum_{S_1 \sqcup \cdots \sqcup S_m = \{1, 2, \ldots, n\}} \frac{|\mu|!}{|\mu_{S_1}|! \cdots |\mu_{S_m}|!} \, \prod_{i=1}^m \Big\langle \, \E_1(\h)^{|\mu_{S_i}|} \, \prod_{k \in \mu_{S_i}} \E_{-k}(0) \, \Big\rangle^\circ
\end{equation}

\cref{alg:disconnected} expresses the correlator on the left as a sum of terms of the form
\[
\langle \, \E_0(b_1 \h) \, \E_0(b_2 \h) \cdots \E_0(b_m \h) \, \rangle = \langle \, \E_0(b_1 \h) \, \rangle \, \langle \, \E_0(b_2 \h) \, \rangle \cdots \langle \, \E_0(b_m \h) \, \rangle.
\]
During this process, each of the $\E$-operators that we start with ultimately contributes to precisely one of the $\langle \, \E_0(b_i \h) \, \rangle$ factors. In particular, a commutation term is never taken between operators that contribute to different factors. It follows that each $\langle \, \E_0(b_i \h) \, \rangle$ is itself a connected correlator for the product of $\E$-operators that contribute to it. \Cref{eq:Einclusionexclusion} is simply a reflection of this fact. Finally, the combinatorial factor arises from a choice of which $\E_1(\h)$ operators in the disconnected correlator on the left contribute to which connected correlator on the right.
\end{proof}

It follows from \cref{prop:conhurwitzwedge} that Steps~(2) and~(3) of \cref{alg:disconnected} can be adapted to give the following algorithm for connected Hurwitz numbers.

\begin{algorithm}[Connected Hurwitz numbers] \label{alg:connected}
Let $a_1, a_2, \ldots, a_m$ be integers and $b_1, b_2, \ldots, b_m$ be non-negative integers such that if $a_i \geq 0$, then $b_i > 0$. Connected correlators of the form $\langle \, \E_{a_1}(b_1\h) \, \E_{a_2}(b_2\h) \cdots \E_{a_m}(b_m\h) \, \rangle^\circ$ can be evaluated by the following algorithm.
\begin{enumerate}[label={(\arabic*)}]
\item If $a_1 + a_2 + \cdots + a_m \neq 0$, return 0.
\item If $a_1 + a_2 + \cdots + a_i \leq 0$ for some $1 \leq i < m$, return 0. \\
(This follows from the fact that if $a_1 + a_2 + \cdots + a_i = 0$ for $1 \leq i < m$, then we have
\[
\langle \, \E_{a_1}(b_1\h) \, \E_{a_2}(b_2\h) \cdots \E_{a_m}(b_m\h) \, \rangle = \langle \, \E_{a_1}(b_1\h) \cdots \E_{a_i}(b_i\h) \, \rangle \, \langle \, \E_{a_{i+1}}(b_{i+1}\h) \cdots \E_{a_m}(b_m\h) \, \rangle,
\]
which does not contribute to the connected correlator.)
\item If $a_1 = a_2 = \cdots = a_m = 0$, return
\[
\delta_{m,1} \frac{1}{[b_1]}.
\]
\item Otherwise, let $\E_{a_i}(b_i\h)$ be the rightmost $\E$-operator with positive subscript and return
\begin{align*}
& \langle \, \E_{a_1}(b_1\h) \cdots \E_{a_{i+1}}(b_{i+1}\h) \, \E_{a_i}(b_i\h) \cdots \E_{a_m}(b_m\h) \, \rangle^\circ \\
& \qquad + [a_ib_{i+1}-a_{i+1}b_i] \, \langle \, \E_{a_1}(b_1\h) \cdots \E_{a_i+a_{i+1}}((b_i+b_{i+1})\h) \cdots \E_{a_m}(b_m\h) \, \rangle^\circ.
\end{align*}
\end{enumerate}
\end{algorithm}

\begin{example} \label{ex:concorrelator}
We apply \cref{alg:connected} to calculate the generating function $F_{211}(\h)$. Any term that is zero due to Step~(2) in the algorithm has been crossed out.
\begin{align*}
F_{211}(\h) ={}& \langle \, \E_1(\h) \, \E_1(\h) \, \E_1(\h) \, \E_1(\h) \, \E_{-2}(0) \, \E_{-1}(0) \, \E_{-1}(0) \, \rangle \\
={}& \langle \, \E_1(\h) \, \E_1(\h) \, \E_1(\h) \, \E_{-2}(0) \, \E_1(\h) \, \E_{-1}(0) \, \E_{-1}(0) \, \rangle + [2] \, \langle \, \E_1(\h) \, \E_1(\h) \, \E_1(\h) \, \E_{-1}(\h) \, \E_{-1}(0) \, \E_{-1}(0) \, \rangle \\
={}& \hcancel{\langle \, \E_1(\h) \, \E_1(\h) \, \E_1(\h) \, \E_{-2}(0) \, \E_{-1}(0) \, \E_1(\h) \, \E_{-1}(0) \, \rangle} + [1] \, \langle \, \E_1(\h) \, \E_1(\h) \, \E_1(\h) \, \E_{-2}(0) \, \E_0(\h) \, \E_{-1}(0) \, \rangle \\
& + [2] \, \langle \, \E_1(\h) \, \E_1(\h) \, \E_{-1}(\h) \, \E_1(\h) \, \E_{-1}(0) \, \E_{-1}(0) \, \rangle + [2]^2 \, \langle \, \E_1(\h) \, \E_1(\h) \, \E_0(2\h) \, \E_{-1}(0) \, \E_{-1}(0) \, \rangle \\
={}& \hcancel{[1] \, \langle \, \E_1(\h) \, \E_1(\h) \, \E_{-2}(0) \, \E_1(\h) \, \E_0(\h) \, \E_{-1}(0) \, \rangle} + [2] \, [1] \, \langle \, \E_1(\h) \, \E_1(\h) \, \E_{-1}(\h) \, \E_0(\h) \, \E_{-1}(0) \, \rangle \\
& + \hcancel{[2] \, \langle \, \E_1(\h) \, \E_1(\h) \, \E_{-1}(\h) \, \E_{-1}(0) \, \E_1(\h) \, \E_{-1}(0) \, \rangle} + [2] \, [1] \, \langle \, \E_1(\h) \, \E_1(\h) \, \E_{-1}(\h) \, \E_0(\h) \, \E_{-1}(0) \, \rangle \\
& + [2]^2 \, \langle \, \E_1(\h) \, \E_0(2\h) \, \E_1(\h) \, \E_{-1}(0) \, \E_{-1}(0) \, \rangle + [2]^3 \, \langle \, \E_1(\h) \, \E_1(3\h) \, \E_{-1}(0) \, \E_{-1}(0) \, \rangle \displaybreak[0] \\
={}& \hcancel{2 \, [2] \, [1] \, \langle \, \E_1(\h) \, \E_{-1}(\h) \, \E_1(\h) \, \E_0(\h) \, \E_{-1}(0) \, \rangle} + 2 \, [2]^2 \, [1] \, \langle \, \E_1(\h) \, \E_0(2\h) \, \E_0(\h) \, \E_{-1}(0) \, \rangle \\
& + \hcancel{[2]^2 \, \langle \, \E_1(\h) \, \E_0(2\h) \, \E_{-1}(0) \, \E_1(\h) \, \, \E_{-1}(0) \, \rangle} + [2]^2 \, [1]\, \langle \, \E_1(\h) \, \E_0(2\h) \, \E_0(\h) \, \E_{-1}(0) \, \rangle \\
& + \hcancel{[2]^3 \, \langle \, \E_1(\h) \, \E_{-1}(0) \, \E_1(3\h) \, \E_{-1}(0) \, \rangle} + [3] \, [2]^3 \, \langle \, \E_1(\h) \, \E_0(3\h) \, \E_{-1}(0) \, \rangle \\
={}& \hcancel{3 \, [2]^2 \, [1] \, \langle \, \E_0(2\h) \, \E_1(\h) \, \E_0(\h) \, \E_{-1}(0) \, \rangle} + 3 \, [2]^3 \, [1] \, \langle \, \E_1(3\h) \, \E_0(\h) \, \E_{-1}(0) \, \rangle \\
& + \hcancel{[3] \, [2]^3 \, \langle \, \E_0(3\h) \, \E_1(\h) \, \E_{-1}(0) \, \rangle} + [3]^2 \, [2]^3 \, \langle \, \E_1(4\h) \, \E_{-1}(0) \, \rangle \\
={}& \hcancel{3 \, [2]^3 \, [1] \, \langle \, \E_0(\h) \, \E_1(3\h) \, \E_{-1}(0) \, \rangle} + 3 \, [2]^3 \, [1]^2 \, \langle \, \E_1(4\h) \, \E_{-1}(0) \, \rangle \\
& + \hcancel{[3]^2 \, [2]^3 \, \langle \, \E_{-1}(0) \, \E_1(4\h) \, \rangle} + [4] \, [3]^2 \, [2]^3 \, \langle \, \E_0(4\h) \, \rangle \\
={}& \hcancel{3 \, [2]^3 \, [1]^2 \, \langle \, \E_{-1}(0) \, \E_1(4\h) \, \rangle} + 3 \, [4] \, [2]^3 \, [1]^2 \, \langle \, \E_0(4\h) \, \rangle + [3]^2 \, [2]^3 \\
={}& 3 \, [2]^3 \, [1]^2 + [3]^2 \, [2]^3 
\end{align*}

It is a simple matter to extract the coefficient from $F_{211}(\h)$ corresponding to the Hurwitz number $H_{g;211}$.
\begin{align*}
H_{g;211} &= \frac{(2g+5)!}{4! \times 2 \times 1 \times 1} \, [\h^{2g+5}] \big( 3 \, [2]^3 \, [1]^2 + [3]^2 \, [2]^3 \big) \\
&= \frac{(2g+5)!}{48} \, [\h^{2g+5}] \Big( 3 \, (e^{\h} - e^{-\h})^3 \, (e^{\h/2} - e^{-\h/2})^2 + (e^{3\h/2} - e^{-3\h/2})^2 \, (e^\h - e^{-\h})^3 \Big) \\
&= \frac{(2g+5)!}{48} \, [\h^{2g+5}] \Big( (e^{6\h} - e^{-6\h}) - 8(e^{3\h} -e^{-3\h}) - 3(e^{2\h} - e^{-2\h}) + 24(e^{\h} - e^{-\h}) \Big) \\
&= \frac{(2g+5)!}{48} \, [\h^{2g+5}] \sum_{k=0}^\infty \frac{2}{(2k+1)!} \Big( (6\h)^{2k+1} - 8 \cdot (3\h)^{2k+1} - 3 \cdot (2\h)^{2k+1} + 24 \cdot (1\h)^{2k+1} \Big) \\
&= \frac{1}{24} \, (6^{2g+5} - 8 \cdot 3^{2g+5} - 3 \cdot 2^{2g+5} + 24 \cdot 1^{2g+5})
\end{align*}
\end{example}

\cref{alg:disconnected,alg:connected} are rather explicit and can be easily implemented in a computer algebra system. The data in \cref{tab:correlators} was produced using SageMath, with higher degree correlators omitted only for brevity~\cite{sagemath}.

\begin{table}[ht!]
\centering
\begin{tabularx}{\textwidth}{cX} \toprule
$\mu$ & $\langle \, \E_1(\h)^{|\mu|} \, \E_{-\mu_1}(0) \, \E_{-\mu_2}(0) \cdots \E_{-\mu_n}(0) \, \rangle^\circ$ \\ \midrule
$2$ & $[2]$ \\
$11$ & $[1]^2$ \\ \midrule
$3$ & $[3]^2$ \\
$21$ & $[2]^3$ \\
$111$ & $[2]^2 \, [1]^2 + 2 \, [1]^4$ \\ \midrule
$4$ & $[4]^3$ \\
$31$ & $[3]^4$ \\
$22$ & $[6] \, [2]^3 + 3 \, [2]^4$ \\
$211$ & $[3]^2 \, [2]^3 + 3 \, [2]^3 \, [1]^2$ \\
$1111$ & $[3]^2 \, [2]^2 \, [1]^2 + 2 \, [3]^2 \, [1]^4 + 9 \, [2]^2 \, [1]^4 + 6 \, [1]^6$ \\ \midrule
$5$ & $[5]^4$ \\
$41$ & $[4]^5$ \\
$32$ & $[8] \, [3]^4 + 4 \, [3]^4 \, [2]$ \\
$311$ & $[4]^2 \, [3]^4 + 4 \, [3]^4 \, [1]^2$ \\
$221$ & $[6] \, [4]^2 \, [2]^3 + 3 \, [4]^2 \, [2]^4 + 6 \, [2]^6$ \\
$2111$ & $[4]^2 \, [3]^2 \, [2]^3 + 3 \, [4]^2 \, [2]^3 \, [1]^2 + 8 \, [3]^2 \, [2]^3 \, [1]^2 + 6 \, [2]^5 \, [1]^2 + 12 \, [2]^3 \, [1]^4$ \\
$11111$ & $[4]^2 \, [3]^2 \, [2]^2 \, [1]^2 + 2 \, [4]^2 \, [3]^2 \, [1]^4 + 9 \, [4]^2 \, [2]^2 \, [1]^4 + 6 \, [4]^2 \, [1]^6 + 16 \, [3]^2 \, [2]^2 \, [1]^4 + 32 \, [3]^2 \, [1]^6 $ \\
 & $+ 18 \, [2]^4 \, [1]^4 + 72 \, [2]^2 \, [1]^6 + 24 \, [1]^8$ \\ \bottomrule
\end{tabularx}
\caption{The connected correlator $\langle \, \E_1(\h)^{|\mu|} \, \E_{-\mu_1}(0) \, \E_{-\mu_2}(0) \cdots \E_{-\mu_n}(0) \, \rangle^\circ$ for each $2 \leq |\mu| \leq 5$, expressed as a polynomial in the $q$-integers according to \cref{alg:connected}.}
\label{tab:correlators}
\end{table}

\subsection{Proof of the main theorem}

The calculations of \cref{ex:concorrelator} already demonstrate the broad features of the proof of~\cref{thm:main}. However, we require certain details regarding the combinatorics of connected correlators, which we present as the following two lemmas.

\begin{lemma} \label{lem:terms}
Let $a_1, a_2, \ldots, a_m$ be integers that sum to $0$ and let $b_1, b_2, \ldots, b_m$ be non-negative integers that sum to~$b$ such that if $a_i \geq 0$, then $b_i > 0$. For $m \geq 2$, the connected correlator $\langle \, \E_{a_1}(b_1\h) \, \E_{a_2}(b_2\h) \cdots \E_{a_m}(b_m\h) \, \rangle^\circ$ can be expressed as a sum of finitely many terms of the form
\[
[k_1] \, [k_2] \cdots [k_{m-2}] \, \frac{[ab]}{[b]},
\]
for positive integers $k_1, k_2, \ldots, k_{m-2}$ and a positive integer $a \leq \max\{a_1, a_2, \ldots, a_m\}$.
\end{lemma}

\begin{proof}
\cref{alg:connected} calculates the connected correlator by iteratively applying the commutator relation in Step~(4), in which a product of two $\E$-operators
\[
\E_{a_i}(b_i\h) \, \E_{a_{i+1}}(b_{i+1}\h)
\]
with $a_i > 0$ and $a_{i+1} \leq 0$ is replaced by the two terms
\[
\E_{a_{i+1}}(b_{i+1}\h) \, \E_{a_i}(b_i\h) + [a_ib_{i+1}-a_{i+1}b_i] \, \E_{a_i+a_{i+1}}((b_i+b_{i+1})\h).
\]
Observe that the conditions on the coefficients appearing here force the inequality $a_ib_{i+1}-a_{i+1}b_i > 0$.

During the process of applying \cref{alg:connected}, new terms are introduced that involve connected correlators multiplied by $q$-integer factors. These connected correlators continue to satisfy the conditions of the lemma.

One can observe that after the application of Step~(4) of \cref{alg:connected},
\begin{itemize}
\item the maximum subscript of an $\E$-operator cannot increase, 
\item the sum of the arguments in any product of $\E$-operators is invariant, and
\item the ``degree'' of each term, which measures the number of $q$-integers plus the number of $\E$-operators, is invariant.
\end{itemize}

It follows that the connected correlator $\langle \, \E_{a_1}(b_1\h) \, \E_{a_2}(b_2\h) \cdots \E_{a_m}(b_m\h) \, \rangle^\circ$ can be expressed as a sum of terms of the form
\[
[k_1] \, [k_2] \cdots [k_{m-2}] \, \langle \, \E_a(b'\h) \, \E_{-a}(b''\h) \, \rangle,
\]
for positive integers $k_1, k_2, \ldots, k_{m-2}$, a positive integer $a \leq \max\{a_1, a_2, \ldots, a_m\}$ and integers $b' + b'' = b$. Now applying the commutation relation of \cref{lem:Ecommutation} to this expression produces that appearing in the statement of the lemma, as desired.
\end{proof}

In order to prove \cref{thm:main}, it is necessary to have some control over the expressions 
\[
[k_1] \, [k_2] \cdots [k_\ell] \, \langle \, \E_{a_1}(b_1\h) \, \E_{a_2}(b_2\h) \cdots \E_{a_m}(b_m\h) \, \rangle
\]
that arise while applying \cref{alg:connected}. We do this by assigning a {\em score} to such an expression, defined to be the integer
\[
(k_1 + k_2 + \cdots + k_\ell) - (b_1 + b_2 + \cdots + b_m) + \sum_{1 \leq i < j \leq m} (a_i b_j - a_j b_i).
\]

\begin{lemma} \label{lem:score}
Let $a_1, a_2, \ldots, a_m$ be integers that sum to $0$ and let $b_1, b_2, \ldots, b_m$ be non-negative integers such that if $a_i \geq 0$, then $b_i > 0$. Consider the process of applying \cref{alg:connected} to compute the connected correlator $\langle \, \E_{a_1}(b_1\h) \, \E_{a_2}(b_2\h) \cdots \E_{a_m}(b_m\h) \, \rangle^\circ$. In each application of Step~(4) of the algorithm to a term with score $s$, the commutation term also has score $s$, while the passing term has score $s - 2t$ for some positive integer $t$. In particular, at any stage of the algorithm, each term produced has a score that is at most the score of the original connected correlator, and with the same parity.
\end{lemma}

\begin{proof}
Consider the expression $[k_1] \, [k_2] \cdots [k_\ell] \, \langle \, \E_{a_1}(b_1\h) \, \E_{a_2}(b_2\h) \cdots \E_{a_m}(b_m\h) \, \rangle$, whose score is
\[
(k_1 + k_2 + \cdots + k_\ell) - (b_1 + b_2 + \cdots + b_m) + \sum_{1 \leq i < j \leq m} (a_i b_j - a_j b_i).
\]
Applying Step~(4) of \cref{alg:connected} produces a commutation term
\[
[k_1] \, [k_2] \cdots [k_\ell] \, [a_ib_{i+1}-a_{i+1}b_i] \, \langle \, \E_{a_1}(b_1\h) \cdots \E_{a_i+a_{i+1}}((b_i+b_{i+1})\h) \cdots \E_{a_m}(b_m\h) \, \rangle^\circ,
\]
which increases the score by
\begin{multline*}
(a_ib_{i+1}-a_{i+1}b_i) + \sum_{k=1}^{i-1} \big( a_k(b_i+b_{i+1}) - (a_i+a_{i+1})b_k \big)
+ \sum_{k=i+2}^m \big( (a_i+a_{i+1}) b_k - a_k (b_i+b_{i+1}) \big) \\
- \sum_{k=1}^{i-1} \big( (a_kb_i - a_ib_k) + (a_kb_{i+1} - a_{i+1}b_k) \big) - (a_ib_{i+1}-a_{i+1}b_i) - \sum_{k=i+2}^m \big( (a_ib_k - a_kb_i) + (a_{i+1}b_k - a_kb_{i+1}) \big) = 0,
\end{multline*}
and a passing term 
\[
[k_1] \, [k_2] \cdots [k_\ell] \, \langle \, \E_{a_1}(b_1\h) \cdots \E_{a_{i+1}}(b_{i+1}\h) \, \E_{a_i}(b_i\h) \cdots \E_{a_n}(b_n\h) \, \rangle^\circ,
\]
which increases the score by
\[
(a_{i+1}b_i - a_ib_{i+1}) - (a_ib_{i+1} - a_{i+1}b_i) = -2 (a_ib_{i+1} - a_{i+1}b_i).
\]
It remains to observe that the conditions of the problem guarantee that $a_ib_{i+1} - a_{i+1}b_i > 0$ and that these conditions persist for all terms created by the algorithm.
\end{proof}

We are now in a position to prove our main theorem, which states that for fixed $\mu = (\mu_1, \ldots, \mu_n)$ with sum $d \geq 2$, the function $H_{g;\mu}$ can be expressed as follows, where $C(\mu, m) \in \mathbb{Z}$.
\[
H_{g;\mu} = \frac{2}{d! \, \mu_1 \cdots \mu_n} \sum_{1 \leq m \leq \binom{d}{2}} C(\mu, m) \cdot m^{d+2g-2+n}
\]

\begin{proof}[Proof of \cref{thm:main}]
Recall from \cref{prop:conhurwitzwedge} that
\begin{equation} \label{eq:fmu1}
F_\mu(\h) = \sum_{g = 0}^\infty H_{g; \mu} \, \frac{\h^{d+2g-2+n}}{(d+2g-2+n)!} = \frac{1}{d! \, \mu_1 \cdots \mu_n} \, \langle \, \E_1(\h)^d \, \E_{-\mu_1}(0) \, \E_{-\mu_2}(0) \cdots \E_{-\mu_n}(0) \, \rangle^\circ.
\end{equation}

By \cref{lem:terms}, the connected correlator here can be expressed as a sum of finitely many terms of the form $[k_1] \, [k_2] \cdots [k_\ell]$, where we set $\ell = d+n-2$. By \cref{lem:score}, the score of any such term must satisfy
\begin{equation} \label{eq:maxparity}
k_1 + k_2 + \cdots + k_\ell \leq d(d-1) \qquad \text{and} \qquad k_1 + k_2 + \cdots + k_\ell \equiv d(d-1) \equiv 0 \pmod{2}.
\end{equation}
The expression $d(d-1)$ here is obtained by computing the score of $\langle \, \E_1(\h)^d \, \E_{-\mu_1}(0) \, \E_{-\mu_2}(0) \cdots \E_{-\mu_n}(0) \, \rangle^\circ$.

Expanding in powers of $\h$ yields
\begin{align*}
[k_1] \, [k_2] \cdots [k_\ell] &= \big( e^{k_1 \h/2} - e^{-k_1 \h/2} \big) \big( e^{k_2 \h/2} - e^{-k_2 \h/2} \big) \cdots \big( e^{k_\ell \h/2} - e^{-k_\ell \h/2} \big) \\
&= \sum_{s_1, s_2, \ldots, s_\ell \in \{\pm 1\}} s_1 s_2 \cdots s_\ell \, e^{(s_1 k_1 + s_2 k_2 + \cdots + s_\ell k_\ell)\h/2}.
\end{align*}

Now pair the terms in this summation that have opposite choices of signs for $s_1, s_2, \ldots, s_\ell$. If $\ell$ is odd, then $[k_1] \, [k_2] \cdots [k_\ell]$ is a sum of terms of the form
\begin{equation} \label{eq:exponentials}
\pm (e^{m\h} - e^{-m\h}),
\end{equation}
where $m = \frac{s_1 k_1 + s_2 k_2 + \cdots + s_\ell k_\ell}{2}$ for some $s_1, s_2, \ldots, s_\ell \in \{ \pm 1\}$.
We know from \cref{eq:maxparity} that
\[
m = \frac{s_1 k_1 + s_2 k_2 + \cdots + s_\ell k_\ell}{2} \leq \frac{k_1 + k_2 + \cdots + k_\ell}{2} \leq \frac{d(d-1)}{2}
\]
and that
\[
m = \frac{s_1 k_1 + s_2 k_2 + \cdots + s_\ell k_\ell}{2} \in \mathbb{Z}.
\]
In fact, we can take $m$ to be a positive integer, since $e^{m\h} - e^{-m\h}$ is odd and vanishes for $m = 0$.

Hence, there exist integers $C(\mu, 1), C(\mu, 2), \ldots, C(\mu, \binom{d}{2})$ such that 
\begin{align} \label{eq:fmu2}
F_\mu(\h) &= \frac{1}{d! \, \mu_1 \cdots \mu_n} \sum_{1 \leq m \leq \binom{d}{2}} C(\mu, m) \, (e^{m\h} - e^{-m\h}) \notag \\
&= \frac{1}{d! \, \mu_1 \cdots \mu_n} \sum_{1 \leq m \leq \binom{d}{2}} C(\mu, m) \sum_{k=0}^\infty \frac{2}{(2k+1)!} \, m^{2k+1} \h^{2k+1} \notag \\
&= \frac{1}{d! \, \mu_1 \cdots \mu_n} \sum_{k=0}^\infty \frac{2}{(2k+1)!} \sum_{1 \leq m \leq \binom{d}{2}} C(\mu, m) \, m^{2k+1} \h^{2k+1}.
\end{align}
Comparing coefficients of powers of $\h$ in \cref{eq:fmu1,eq:fmu2}, we have
\begin{equation} \label{eq:structure}
H_{g;\mu} = \frac{2}{d! \, \mu_1 \cdots \mu_n} \sum_{1 \leq m \leq \binom{d}{2}} C(\mu, m) \cdot m^{d+2g-2+n}.
\end{equation}

This essentially completes the proof of the theorem, but only under the assumption that $\ell = d+n-2$ is odd. However, the case when $\ell$ is even can be handled in an entirely analogous manner. The only difference is that \cref{eq:exponentials} changes to $\pm (e^{m\h} + e^{-m\h})$ and we must now allow the possibility that $m = 0$. This contributes to \cref{eq:structure} only if $d+2g-2+n = 0$, in which case the Hurwitz number must be $H_{0;1} = 1$. Note that we have excluded this case from the statement of the theorem by requiring $d \geq 2$.
\end{proof}

The coefficients $C(\mu,m)$ appearing in \cref{thm:main} can be easily extracted from \cref{alg:connected} using a computer algebra system. The data in \cref{tab:coefficients} was produced using SageMath, with higher degree coefficients omitted only for brevity~\cite{sagemath}.

\begin{table}[ht!]
\centering
\begin{tabularx}{\textwidth}{>{\centering\arraybackslash}X*{10}{Y}} \toprule
$\mu$ & \multicolumn{10}{c}{$m$} \\ \cmidrule{2-11}
 & $1$ & $2$ & $3$ & $4$ & $5$ & $6$ & $7$ & $8$ & $9$ & $10$ \\ \midrule
$2$ & $1$ \\
$11$ & $1$ \\ \midrule
$3$ & $0$ & $0$ & $1$ \\
$21$ & $-3$ & $0$ & $1$ \\
$111$ & $-9$ & $0$ & $1$ \\ \midrule
$4$ & $0$ & $-3$ & $0$ & $0$ & $0$ & $1$ \\
$31$ & $0$ & $0$ & $-4$ & $0$ & $0$ & $1$ \\
$22$ & $0$ & $-9$ & $0$ & $0$ & $0$ & $1$ \\
$211$ & $24$ & $-3$ & $-8$ & $0$ & $0$ & $1$ \\
$1111$ & $144$ & $-9$ & $-16$ & $0$ & $0$ & $1$ \\ \midrule
$5$ & $0$ & $0$ & $0$ & $0$ & $-4$ & $0$ & $0$ & $0$ & $0$ & $1$ \\
$41$ & $0$ & $10$ & $0$ & $0$ & $0$ & $-5$ & $0$ & $0$ & $0$ & $1$ \\
$32$ & $20$ & $15$ & $0$ & $-10$ & $-4$ & $0$ & $0$ & $0$ & $0$ & $1$ \\
$311$ & $20$ & $-15$ & $40$ & $-10$ & $4$ & $-10$ & $0$ & $0$ & $0$ & $1$ \\
$221$ & $0$ & $100$ & $0$ & $-20$ & $0$ & $-5$ & $0$ & $0$ & $0$ & $1$ \\
$2111$ & $-400$ & $120$ & $120$ & $-40$ & $8$ & $-15$ & $0$ & $0$ & $0$ & $1$ \\
$11111$ & $-4000$ & $600$ & $400$ & $-100$ & $16$ & $-25$ & $0$ & $0$ & $0$ &$1$ \\ \bottomrule
\end{tabularx}
\caption{The coefficients $C(\mu, m)$ for $2 \leq |\mu| \leq 5$ and $1 \leq m \leq |\mu| (|\mu|-1) / 2$, which appear in the statement of \cref{thm:main}.}
\label{tab:coefficients}
\end{table}

\section{Observations, generalisations and variations} \label{sec:extras}

In this section, we prove a result on the large genus asymptotics of Hurwitz numbers that follows from our main theorem. We pose the question of whether other natural enumerative problems possess similar structure to that of Hurwitz numbers for fixed ramification profile and varying genus. As an example, we replicate the analysis of the previous section to prove that orbifold Hurwitz numbers are governed by the same structure. We then provide numerical evidence to support the conjecture that monotone Hurwitz numbers can be expressed as a linear combination of exponentials plus a linear term.

\subsection{Large genus asymptotics}

The large genus asymptotics for various problems in enumerative geometry has been of significant interest in recent years~\cite{agg21, EGGGL23}. \cref{thm:main} on the structure of Hurwitz numbers leads to \cref{thm:asymptotics} on their large genus asymptotics. It states that for fixed $\mu = (\mu_1, \ldots, \mu_n)$ with sum $d$, we have
\[
H_{g;\mu} \sim \frac{2}{d! \, \mu_1 \cdots \mu_n} \binom{d}{2}^{d+2g-2+n} \quad \text{as } g \to \infty.
\]

\begin{proof}[Proof of \cref{thm:asymptotics}]
From \cref{thm:main}, we have the following equation, so the theorem follows immediately if we can show that $C(\mu, \binom{d}{2}) = 1$.
\[
H_{g;\mu} = \frac{2}{d! \, \mu_1 \cdots \mu_n} \sum_{1 \leq m \leq \binom{d}{2}} C(\mu, m) \cdot m^{d+2g-2+n},
\]

From \cref{lem:score}, we know that applying \cref{alg:disconnected} to
$\langle \, \E_1(\h)^d \, \E_{-\mu_1}(0) \, \E_{-\mu_2}(0) \cdots \E_{-\mu_n}(0) \, \rangle
$ produces terms with score at most $d(d-1)$. Furthermore, a term with score equal to $d(d-1)$ can only be achieved by always taking the commutation term in Step~(4) of the algorithm, since the passing term strictly reduces the score. In the process of calculating these commutation terms iteratively, one only encounters products of $\E$-operators in which those with positive subscripts appear to the left of those with non-positive subscripts. It follows that at the conclusion of the algorithm, there is only one term with score equal to $d(d-1)$ and that it contributes to the connected correlator $\langle \, \E_1(\h)^d \, \E_{-\mu_1}(0) \, \E_{-\mu_2}(0) \cdots \E_{-\mu_n}(0) \, \rangle^\circ$. This term is necessarily of the form $[k_1] \, [k_2] \cdots [k_\ell]$ for $k_1 + k_2 + \cdots + k_\ell = d(d-1)$.

From the proof of \cref{thm:main}, we see that $C(\mu, \binom{d}{2})$ is equal to a signed sum of the coefficients of $[k_1] \, [k_2] \cdots [k_\ell]$ for $k_1 + k_2 + \cdots + k_\ell = d(d-1)$, appearing in the connected correlator. From the previous discussion, we deduce that $C(\mu, \binom{d}{2}) = \pm 1$, but the sign must be positive as Hurwitz numbers are positive. So it is indeed the case that $C(\mu, \binom{d}{2}) = 1$ and we have the large genus asymptotics as stated.
\end{proof}

The approximation given by the asymptotic formula of \cref{thm:asymptotics} is rather precise. For example, the ratio of $H_{3;32}$ to its approximation is $1.0023778\cdots$, while the ratio of $H_{20;32}$ to its approximation is $1.00000000000011369\cdots$.

The data of \cref{tab:coefficients} suggests that we not only have $C(\mu, \binom{d}{2}) = 1$, but also $C(\mu, m) = 0$ for $\binom{d-1}{2} < m < \binom{d}{2}$. This was previously known in the case $\mu = (1, 1, \ldots, 1)$ considered by Hurwitz~\cite{hur1891, dub-yan-zag17}. We conjecture that this holds more generally, which allows for a more refined statement on the large genus asymptotics.

\begin{conjecture}
For fixed $\mu = (\mu_1, \ldots, \mu_n)$ with sum $d \geq 2$, we have $C(\mu, m) = 0$ for $\binom{d-1}{2} < m < \binom{d}{2}$. It then follows that
\[
H_{g;\mu} = \frac{2}{d! \, \mu_1 \cdots \mu_n} \binom{d}{2}^{d+2g-2+n} + O \Bigg( \binom{d-1}{2}^{d+2g-2+n} \Bigg) \quad \text{as } g \to \infty.
\]
\end{conjecture}

\subsection{Orbifold Hurwitz numbers} \label{subsec:orbifold}

Many properties of Hurwitz numbers carry over to the generalisation known as orbifold Hurwitz numbers, which arise in the enumerative geometry of the orbifold $\mathbb{CP}^1[r]$~\cite{BHLM14, do-lei-nor16, joh-pan-tse11}. The proof of \cref{thm:main} can be adapted to prove an analogous structure theorem for orbifold Hurwitz numbers.

\begin{definition}
Fix a positive integer $r$. Let $g$ be an integer and $\mu = (\mu_1, \ldots, \mu_n)$ be a tuple of positive integers. The {\em $r$-orbifold Hurwitz number} $H^{[r]}_{g;\mu}$ is the weighted enumeration of genus $g$ connected branched covers of $\mathbb{CP}^1$ with ramification profile $\mu$ over $\infty \in \mathbb{CP}^1$, ramification profile $(r, r, \ldots, r)$ over $0 \in \mathbb{CP}^1$, simple ramification over $\frac{|\mu|}{r} + 2g - 2 + n$ prescribed points of $\mathbb{CP}^1$, and no branching elsewhere. The preimages of $\infty \in \mathbb{CP}^1$ are labelled from $1$ to $n$ such that the point labelled $i$ has ramification index~$\mu_i$.
\end{definition}

Observe that the $r$-orbifold Hurwitz number $H^{[r]}_{g;\mu}$ must be zero unless $|\mu|$ is divisible by $r$. Again, by a standard monodromy argument using the Riemann existence theorem, one can express orbifold Hurwitz numbers as an equivalent enumeration of permutations in the symmetric group, although we omit the statement here for brevity. Orbifold Hurwitz numbers can be computed using the infinite wedge space via the following analogue of \cref{prop:conhurwitzwedge}.

\begin{proposition} \label{prop:orbifoldwedge}
For fixed $\mu = (\mu_1, \ldots, \mu_n)$ with sum $d$ that is divisible by $r$, the exponential generating function for $r$-orbifold Hurwitz numbers
\[
F^{[r]}_\mu(\h) = \sum_{g = 0}^\infty H^{[r]}_{g;\mu} \, \frac{\h^{\frac{d}{r}+2g-2+n}}{(\frac{d}{r}+2g-2+n)!}
\]
can be expressed as the connected correlator
\[
F^{[r]}_\mu(\h) = \frac{1}{r^{d/r} (d/r)! \, \mu_1 \cdots \mu_n} \, \langle \, \E_r(r\h)^{d/r} \, \E_{-\mu_1}(0) \cdots \E_{-\mu_n}(0) \, \rangle^\circ.
\]
\end{proposition}

The connected correlator of $\E$-operators appearing in \cref{prop:orbifoldwedge} is amenable to the same techniques used in \cref{sec:proof} and we obtain the following.

\begin{theorem} \label{thm:orbifold}
For fixed $\mu = (\mu_1, \ldots, \mu_n)$ with sum $d$ that is divisible by $r$, the function $H^{[r]}_{g;\mu}$ can be expressed as follows, where $C^{[r]}(\mu, m) \in \mathbb{Z}$.
\[
H^{[r]}_{g;\mu} = \frac{2}{r^{d/r} (d/r)! \, \mu_1 \cdots \mu_n} \sum_{1 \leq m \leq d(d-1)/2} C^{[r]}(\mu, m) \cdot m^{\frac{d}{r}+2g-2+n}
\]
\end{theorem}

\begin{proof}
The argument mimics that used in the previous section to prove \cref{thm:main}. By \cref{lem:terms}, the connected correlator $\langle \, \E_r(r\h)^{d/r} \, \E_{-\mu_1}(0) \cdots \E_{-\mu_n}(0) \, \rangle^\circ$ can be expressed as a sum of terms of the form
\[
[k_1] \, [k_2] \cdots [k_\ell] \, \frac{[ad]}{[d]},
\]
where $\ell = \frac{d}{r} + n - 2$ and $0 < a \leq r$. The score of the correlator is $d(d-1)$, so we must have
\[
k_1 + k_2 + \cdots + k_\ell + (ad-d) \leq d(d-1) \qquad \text{and} \qquad k_1 + k_2 + \cdots + k_\ell + (ad-d) \equiv d(d-1) \equiv 0 \pmod{2}.
\]

Expanding in powers of $\h$ yields
\begin{align*}
& [k_1] \, [k_2] \cdots [k_\ell] \, \frac{[ad]}{[d]} \\
={}& \big( e^{k_1 \h/2} - e^{-k_1 \h/2} \big) \big( e^{k_2 \h/2} - e^{-k_2 \h/2} \big) \cdots \big( e^{k_\ell \h/2} - e^{-k_\ell \h/2} \big) \big( e^{(a-1)d\h/2} + e^{(a-3)d\h/2} + \cdots + e^{-(a-1)d\h/2} \big) \\
={}& \sum_{i=0}^a \sum_{s_1, s_2, \ldots, s_\ell \in \{\pm 1\}} s_1 s_2 \cdots s_\ell \, e^{(s_1 k_1 + s_2 k_2 + \cdots + s_\ell k_\ell)\h/2} e^{(a-1-2i)d\h/2}.
\end{align*}

Now pair the terms in this summation that have opposite choices of signs for $s_1, s_2, \ldots, s_\ell$ and values of $i$ that sum to $a$. If $\ell$ is odd, then $[k_1] \, [k_2] \cdots [k_\ell] \, \frac{[ad]}{[d]}$ is a sum of terms of the form $\pm (e^{m\h} - e^{-m\h})$, while if $\ell$ is even, it is a sum of terms of the form $\pm (e^{m\h} + e^{-m\h})$.

From this point onwards, the proof is essentially identical to that of \cref{thm:main} and we omit the details.
\end{proof}

\cref{thm:orbifold} allows us to obtain the large genus asymptotics for orbifold Hurwitz numbers, using the same argument as in the proof of \cref{thm:asymptotics}.

\begin{theorem}
For fixed $\mu = (\mu_1, \ldots, \mu_n)$ with sum $d$ that is divisible by $r$, we have
\[
H^{[r]}_{g;\mu} \sim \frac{2}{r^{d/r} (d/r)! \, \mu_1 \cdots \mu_n} \binom{d}{2}^{\frac{d}{r}+2g-2+n} \quad \text{as } g \to \infty.
\]
\end{theorem}

\subsection{Monotone Hurwitz numbers and beyond}

It is natural to seek other enumerative problems exhibiting structure analogous to that stated for Hurwitz numbers in \cref{thm:main}. The case of orbifold Hurwitz numbers considered in \cref{subsec:orbifold} provides an example that is amenable to the analysis of connected correlators of $\E$-operators. Another natural candidate is given by the monotone Hurwitz numbers, whose definition is a mild modification of the one for usual Hurwitz numbers. Monotone Hurwitz numbers were introduced by Goulden, Guay-Paquet and Novak, who realised that they arise as coefficients in the large $N$ expansion of the HCIZ matrix integral over the unitary group $U(N)$~\cite{gou-gua-nov14}.

We say that a tuple $((a_1~b_1), (a_2~b_2), \cdots, (a_m~b_m))$ of transpositions is {\em monotone} if $b_1 \leq b_2 \leq \cdots \leq b_m$, where we write $a_i < b_i$ by convention.

\begin{definition}
Let $g$ be an integer and $\mu = (\mu_1, \ldots, \mu_n)$ be a tuple of positive integers. The {\em monotone Hurwitz number} $\vec{H}_{g;\mu}$ is equal to $\frac{|\mathrm{Aut}(\mu)|}{|\mu|!}$ times the number of monotone tuples $(\tau_1, \tau_2, \ldots, \tau_{|\mu|+2g-2+n})$ of transpositions in the symmetric group $S_{|\mu|}$ that generate a transitive subgroup and whose product has cycle type $\mu$. The disconnected monotone Hurwitz number $\vec{H}_{g;\mu}^\bullet$ is obtained by dropping the transitivity constraint.
\end{definition}

Monotone Hurwitz numbers can be effectively computed via the monotone cut-and-join recursion~\cite{gou-gua-nov13}. Thus, one can seek to numerically confirm whether for fixed $\mu$, the function $\vec{H}_{g;\mu}$ can be expressed as a linear combination of exponentials in the quantity $|\mu|+2g-2+n$. The data in \cref{tab:monotone} strongly suggests that this is almost true, except that one must additionally include a linear term. This is stated precisely in the following conjecture, which might also be resolved by using the infinite wedge space to capture monotone Hurwitz numbers~\cite{kra-lew-sha19}.

\begin{table}[ht!]
\centering
\begin{tabular}{cl} \toprule
$\mu = (\mu_1, \ldots, \mu_n)$ & $\vec{H}_{g;\mu}$ for $0 \leq g \leq 20$ ($k = |\mu| + 2g - 2 + n$) \\ \midrule
$2$ & $1 \cdot 1^k$ \\
$11$ & $1 \cdot 1^k$ \\ \midrule
$3$ & $\frac{2}{3} \cdot 2^k - \frac{2}{3} \cdot 1^k$ \\
$21$ & $\frac{2}{3} \cdot 2^k - \frac{4}{3} \cdot 1^k$ \\
$111$ & $\frac{2}{3} \cdot 2^k - \frac{8}{3} \cdot 1^k$ \\ \midrule
$4$ & $\frac{3}{8} \cdot 3^k - \frac{2}{3} \cdot 2^k + \frac{5}{24} \cdot 1^k$ \\
$31$ & $\frac{3}{8} \cdot 3^k - 1 \cdot 2^k + \frac{5}{8} \cdot 1^k$ \\
$22$ & $\frac{3}{8} \cdot 3^k - \frac{2}{3} \cdot 2^k + \frac{7}{24} \cdot 1^k \begingroup \color{red} - \frac{1}{2} \cdot k \endgroup$ \\
$211$ & $\frac{3}{8} \cdot 3^k - \frac{4}{3} \cdot 2^k + \frac{49}{24} \cdot 1^k \begingroup \color{red} - \frac{1}{2} \cdot k \endgroup$ \\
$1111$ & $\frac{3}{8} \cdot 3^k - 2 \cdot 2^k + \frac{61}{8} \cdot 1^k \begingroup \color{red}- \frac{3}{2} \cdot k \endgroup$ \\ \midrule
$5$ & $\frac{8}{45} \cdot 4^k - \frac{9}{20} \cdot 3^k + \frac{14}{45} \cdot 2^k - \frac{7}{180} \cdot 1^k$ \\
$41$ & $\frac{8}{45} \cdot 4^k - \frac{3}{5} \cdot 3^k + \frac{28}{45} \cdot 2^k - \frac{7}{45} \cdot 1^k$ \\
$32$ & $\frac{8}{45} \cdot 4^k - \frac{9}{20} \cdot 3^k - \frac{8}{90} \cdot 2^k + \frac{29}{60} \cdot 1^k \begingroup \color{red}+ \frac{1}{3} \cdot k \endgroup$ \\
$311$ & $\frac{8}{45} \cdot 4^k - \frac{3}{4} \cdot 3^k + \frac{8}{9} \cdot 2^k - \frac{19}{60} \cdot 1^k \begingroup \color{red}+ \frac{1}{3} \cdot k \endgroup$ \\
$221$ & $\frac{8}{45} \cdot 4^k - \frac{3}{5} \cdot 3^k - \frac{8}{45} \cdot 2^k + \frac{3}{5} \cdot 1^k \begingroup \color{red}+ \frac{4}{3} \cdot k \endgroup$ \\
$2111$ & $\frac{8}{45} \cdot 4^k - \frac{9}{10} \cdot 3^k + \frac{32}{45} \cdot 2^k - \frac{83}{30} \cdot 1^k \begingroup \color{red}+ \frac{10}{3} \cdot k \endgroup$ \\
$11111$ & $\frac{8}{45} \cdot 4^k - \frac{6}{5} \cdot 3^k + \frac{64}{45} \cdot 2^k - \frac{122}{5} \cdot 1^k \begingroup \color{red} + \frac{40}{3} \cdot k \endgroup$ \\ \bottomrule
\end{tabular}
\caption{Conjectural formulae for the monotone Hurwitz numbers $\vec{H}_{g;\mu}$ for each $2 \leq |\mu| \leq 5$, which is known to hold for $0 \leq g \leq 20$. Note that each formula is a linear combination of exponentials in $k = |\mu| + 2g - 2 + n$, plus a linear term in $k$ that is shown in red.}
\label{tab:monotone}
\end{table}

\begin{conjecture} \label{con:monotone}
For fixed $\mu = (\mu_1, \ldots, \mu_n)$ with sum $d \geq 2$, the function $\vec{H}_{g;\mu}$ can be expressed as follows, where $\vec{C}(\mu, m) \in \mathbb{Q}$.
\[
\vec{H}_{g;\mu} = \vec{C}(\mu,0) \cdot (d+2g-2+n) + \sum_{m=1}^{d-1} \vec{C}(\mu, m) \cdot m^{d+2g-2+n}
\]
\end{conjecture}

It is interesting to note that the particular structure suggested by \cref{con:monotone} does not extend to the disconnected monotone Hurwitz numbers, since the family of functions of the form ``linear combination of exponentials plus a linear term'' is not closed under multiplication. Although the possibly disconnected enumeration may seem more natural from the algebraic point of view, the connected enumeration is more natural from the geometric point of view and often exhibits better behaviour.

One can continue to seek enumerative problems that exhibit analogous structure, beyond Hurwitz numbers and their monotone counterpart. Dubrovin, Yang and Zagier reproduced \cref{thm:hurwitz} of Hurwitz, but also showed that a certain weighted enumeration of connected graphs with $g$ independent loops and $d$~vertices could be expressed as a linear combination of exponentials~\cite[Theorem~4]{dub-yan-zag17}. One wonders whether a refinement or generalisation may lead to an enumeration $G_{g;\mu}$ with the same structure, where $\mu$ is an arbitrary partition.

Another source of inspiration for amenable enumerative problems is the topological recursion of Chekhov, Eynard and Orantin~\cite{che-eyn06, eyn-ora07}. The topological recursion is known to govern the single, orbifold and monotone Hurwitz numbers already considered in this paper~\cite{eyn-mul-saf11, BHLM14, do-lei-nor16, do-dye-mat17}. Naively, one can think of it as a mechanism to produce enumerative data $X_{g;\mu}$ for $g$ a non-negative integer and $\mu$ a partition, from the initial data of a ``spectral curve''. However, such a brief description belies the ever-increasing power, richness and vastness displayed by the theory. One wonders whether formulae of the type appearing in \cref{thm:main} may govern such enumerative data generated by the topological recursion, for fixed partition and varying genus.

\bibliographystyle{plain}
\bibliography{hurwitz-number-structure}

\end{document}